\documentclass[11pt]{amsart}
\usepackage[latin9]{inputenc}
\usepackage{units}
\usepackage{url}
\usepackage{amsmath}
\usepackage{amsthm}
\usepackage{amssymb}
\usepackage[english]{babel}
\usepackage{verbatim}
\usepackage{ulem}
\usepackage{color}
\usepackage{dsfont}
\usepackage{mathtools}

\makeatother

\usepackage{babel}

\makeatletter
\newtheorem{thm}{Theorem}[section]
\newtheorem{lem}[thm]{Lemma}

\newtheorem{prop}[thm]{Proposition}

\newtheorem{cor}[thm]{Corollary}

\global\long\def\cE{\mathcal{E}}%
\global\long\def\cI{\mathcal{I}}%
\global\long\def\cZ{\mathcal{Z}}%

\newcommand{\Z}{{\mathbb Z}} 

\newcommand{\R}{{\mathbb R}}

\newcommand{\N}{{\mathbb N}}

\renewcommand{\^}{\widehat}

\newcommand{\intinf}{\int_{-\infty}^\infty}

\newcommand{\conv}{*}

 \newcommand{\supp}{\operatorname{supp}}
 
\newcounter{ToDo}

\newcommand{\off}{\operatorname{off}}

\global\long\def\bx{\mathbf{x}}%

\numberwithin{equation}{section}

\begin{document}

\title[The pair correlation of small non-integer powers]{The metric theory of the pair correlation function for small non-integer powers}
\author{Ze\'ev Rudnick and Niclas Technau}

\date{\today}

\address{School of Mathematical Sciences, Tel Aviv University, Tel Aviv 69978, Israel}
 \email{rudnick@tauex.tau.ac.il}
 
\address{Department of Mathematics, University of Wisconsin, 480 Lincoln Drive, Madison, WI, 53706, USA }
  \email{technau@wisc.edu}
 
\thanks{ZR was supported by the European Research Council (ERC) under the European Union's Horizon 2020 research and innovation programme (Grant agreement No. 786758).
NT was supported by the Austrian Science Fund (FWF), project J-4464.}

\begin{abstract} 
 For $0<\theta<1$, we show that for almost all $\alpha$, the pair correlation function of the sequence of fractional parts of $\{\alpha n^\theta:n\geq 1 \}$ is Poissonian.
\end{abstract}
\maketitle

\section{Introduction}

 The theory of uniform distribution of sequences modulo one has a long history, with many developments since its creation over a century ago, see   \cite{KN}. Much more recent is the study of 
  ``local'' statistics  of sequences,   motivated by problems in quantum chaos and the theory of the Riemann zeta function, see e.g. \cite{ZRWhatisQC}. A key example of such a  local  statistic is the {\sl pair correlation function} $R_2(x)$  which, for a sequence of points $\{\theta_n\}\subset \R/\Z$, assigns to  every $s>0$ the limit
\begin{equation}\label{def pair cor} 
\lim_{N\to \infty} \frac 1N \#\{ 1\leq x\neq y\leq N: |\theta_x-\theta_y| \leq \frac{s}{N} \} = \int_{-s}^s R_2(t)dt
 \end{equation}
assuming that the limit exists, which is in itself a major problem. The pair correlation is analytically the most accessible example of a ``local'' statistic, the easiest to visualize being
the nearest neighbour spacing distribution $P(s)$, which is defined as the limit distribution (assuming it exists) of the gaps between neighbouring elements in the sequence, rescaled so as to have mean value unity. A simple model of what to expect is a ``random sequence'' (the Poisson model), say taking  independent uniform random variables. In this case  the nearest neighbour spacing distribution is exponential: $P(s) =\exp(-s)$ and  the pair correlation function is $R_2(t)\equiv 1$: 
\begin{equation}\label{poisson pair cor} 
\lim_{N\to \infty} \frac 1N \#\{ 1\leq x\neq y\leq N: |\theta_x-\theta_y|  \leq \frac{s}{N}   \} = \int_{-s}^s 1 dx =2s  .
 \end{equation}

One of the few cases where the pair correlation of a fixed (deterministic) sequence is known is the  sequence of square roots of integers $\theta_n=\sqrt{n}$.
     It was shown by Elkies and  McMullen \cite{EM} that the nearest neighbour spacing distribution of $\{\sqrt{n}\bmod 1\}$ exists, and is non-Poissonian (this was discovered numerically by Boshernitzan in 1993). In particular, the level spacing distribution is flat near the origin $P(s) = 6/\pi^2$, $0\leq s\leq 1/2$, and is piecewise analytic.
    Surprisingly,  El-Baz, Marklof and Vinogradov \cite{EMV} showed that the pair correlation of the fractional parts of $\{\sqrt{n}:n\neq k^2\}$ is nonetheless Poissonian, after removing perfect squares $n=k^2$ (without this removal the pair correlation blows up).  
   Their work has as its departure point the methods of \cite{EM}, which use homogeneous dynamics, and does not apply to the more general case of $\alpha\sqrt{n}$ when $\alpha^2$ is irrational.
 
Our goal in this paper is to show that the pair correlation of $\alpha \sqrt{n}$ is indeed Poissonian for {\sl  almost all} $\alpha$, in the sense that the exceptional sense has Lebesgue measure zero. In fact, we show more generally
\begin{thm}\label{main thm}
Fix $\theta\in (0,1)$. Then for almost all $\alpha$, the pair correlation function of the sequence $\alpha n^\theta$ is Poissonian. 
\end{thm} 
The corresponding result for integer $\theta>1$ is due to Rudnick and Sarnak \cite{RudSarpair}. 
The case of non-integer $\theta>1$ was recently solved by Aistleitner, El-Baz, and Munsch  \cite{AEM}, and we build on some of their work, though their methods fail to cover $\theta<1$, see \cite[Section 9]{AEM} for an explanation. 
 Very recently, Lutsko,   Sourmelidis and Technau \cite{LT} have succeeded in showing a deterministic result, that for $\theta\leq 1/3$ and \underline{all} $\alpha\neq 0$, the pair correlation of $\alpha n^\theta$ is Poissonian.  

\subsubsection*{Acknowledgments}
The authors thank Daniel El-Baz 
and Shvo Regavim for insightful feedback.

\section{An overview of the proof}
 \subsection{Smoothing}
 
Fix $\theta \in (0,1)$. We study a smooth version of the pair correlation function in which we fix
 $h\in  C^\infty([1,2])$, $h\geq 0$,   and an even $f\in C_{c}^\infty(\R ) $ to define
\begin{equation} \label{def: smooth pair correlation}
R_{f,h}(\alpha, N):=\frac 1N\sum_{  x\neq y \geq 1}
 h\Big(\frac xN\Big)h\Big(\frac yN\Big) 
F_N (\alpha (x^\theta-y^\theta )) ,
\end{equation}
where $F_N$ denotes the $1$-periodised version of $f$ localised to the scale $1/N$,
that is 
$$ 
F_N(x) := \sum_{k\in \Z} f(N(x+k)).
$$

It is a standard argument (see e.g. \cite[\S 8.6.1 and \S 8.6.3]{Marklof}) that a sequence $\bx$ has Poissonian pair correlation, that is \eqref{poisson pair cor}  holds, if and only if
for any $h,f\in C_c^\infty(\R)$ 
\begin{equation}\label{def: smoothed pair correlation function convergence}
\lim_{N\to \infty} R_{f,h}(\alpha, N)  = \intinf f(t) \mathrm{d}t 
\bigg(\intinf h(s) \mathrm{d}s\bigg)^2.
\end{equation}
 Hence we want to show that \eqref{def: smoothed pair correlation function convergence} holds for almost all $\alpha >0$. 
 To fix ideas and simplify the notation, we assume that  $\alpha\in [1,2]$.
 
 Note that choosing $h$ supported in $[1,2]$ rather than $[0,1]$ corresponds to taking $N/2<x,y\leq N$ in \eqref{def pair cor}. 
We do this for convenience, and this kind of modification is natural in statistical physics (to work ``in the bulk'') and one usually 
does not expect it to change any of the local statistics.
A striking exception is for the sequence $\sqrt{n}$ (that is $\theta=1/2$, $\alpha=1$), where the limiting level spacing distribution depends on the choice of $c\in [0,1)$ in choosing $cN<x,y<N $, as was found by  Elkies and McMullen \cite[section 3.5]{EM};  
but the pair correlation (after removing squares) does not, as is pointed out by  El-Baz, Marklof and Vinogradov 
\cite[remark 1, page 5]{EMV}.
Furthermore, it will be convenient to use the notation
\begin{equation*}
\Theta:=\frac 1{1-\theta}.
\end{equation*}
Note that for $\theta\in (0,1)$, we have $\Theta>1$, e.g. $\theta=1/2$ gives $\Theta=2$. 

Throughout the paper, we denote $e(z):=\exp(2\pi \sqrt{-1}z)$, and use a normalized Fourier transform $\^f(x):=\intinf f(y)e(-xy)dy$. 

\bigskip 

We break down our argument into several key steps: 

\subsection{Step 1}
Let 
\begin{equation}\label{def: S,E}
\cE_{N,j}(\alpha) = \sum_{y \geq 1} h\Big(\frac yN\Big) e( \alpha j y^\theta ).
\end{equation}
By the Fourier expansion of $F_N$ and the decay of $\^f$,
 write $R_{f,h}(N)(\alpha)$ as
\begin{multline*}
S(f,h;N)(\alpha) + \intinf f(x)dx \bigg(\intinf h(y)dy\bigg)^2
 - f(0)\intinf h(  y)^2 dy +O(N^{-100})
\end{multline*}
with
$$
S(f,h;N)(\alpha) =\frac {2}{N^2}\sum_{1\leq j \ll N^{1+\epsilon} } \^f\Big(\frac jN\Big) |\cE_{N,j}(\alpha)|^2.
$$

 Now use Poisson summation and stationary phase expansion (van der Corput's ``B-process'') to replace $\cE_{N,j}(\alpha)$ by a shorter sum:
\begin{prop}\label{prop: pruning}
For $j>0$,
\[
\cE_{N,j}(\alpha) = \tilde \cE_{N,j}(\alpha)   + O\left( \frac{N^{1-\frac{\theta}{2}}}{j^{1/2}} \right) 
\]
with
\[
\tilde \cE_{N,j}(\alpha):= c_1 \cdot  (\alpha j)^{\frac \Theta 2} 
\sum_{m   \asymp j/N^{1-\theta}} 
  \frac{1}{ m^{\frac{\Theta+1}2}}  
h\Big(\frac{(\theta\alpha j)^{\Theta}}{Nm^\Theta}\Big)   
e\Big(c_2 \frac{(\alpha j)^{\Theta}}{m^{\Theta-1}} \Big) 
  \]
where $c_1 := \frac{\theta^{\frac 1{2(1-\theta)}}}{\sqrt{1-\theta}} e(-\frac 18)$ and  
$c_2 :=  \theta^{\frac{\theta}{1-\theta}}-\theta^{\frac 1{1-\theta}}$. 
\end{prop}
Note that $ \tilde \cE_{N,j}(\alpha) =0$ if $0<j\ll N^{1-\theta}$.

\subsection{Step 2} Show that $ \tilde \cE_{N,j}(\alpha)$ 
is typically of size $N^{1/2}$ along a polynomially sparse subsequence: 
We define a probability measure $d\mu(\alpha)$ on $[1,2]$ by 
\begin{equation} \label{def of dmu}
d\mu(\alpha) =\frac{ \Theta \rho(\alpha^\Theta)d\alpha}{\alpha}
\end{equation}
where  
$\rho\in C^\infty([1,2])$, $\rho\geq 0$, 
normalized by $\intinf \rho(x)\frac{dx}{x} =1$. 
 
\begin{prop}\label{prop second moment of E}
The estimate
 \[
 \intinf | \tilde \cE_{N,j}(\alpha) |^2 d\mu(\alpha) \ll N 
 \]
holds uniformly for any $0<j<N^2$.
 \end{prop}
Let 
$$ 
N_\ell := \ell^C
$$
where $C\in \N$ is a large constant.
We deduce that on the subsequence 
$$
\mathcal N=\{N_\ell,\ell\geq 1\} 
$$ 
for almost all $\alpha\in [1,2]$
all frequencies $j\gg N^{1-\epsilon}$ (where $\epsilon>0$ is small)
are negligible for $N\in \mathcal N$: 
\begin{cor}\label{cor: truncate j less than Neps}
For almost all $\alpha\in [1,2]$ and all $N_\ell \in \mathcal N$,
we have that
\[
S(f,h;N_\ell)(\alpha) =\frac {2}{N_\ell^2}
\sum_{N_\ell^{1-\epsilon} \ll j\ll N_\ell^{1+\epsilon}} 
\^f\Big(\frac{j}{N_\ell}\Big) |\cE_{N_\ell,j}(\alpha)|^2 +O(N_\ell^{-\epsilon/2}),
\]
and 
\begin{equation}\label{eq: crude upper bound}
\frac {2}{N_\ell^2}
\sum_{j\ll N_\ell^{1+\epsilon}} 
\^f\Big(\frac{j}{N_\ell}\Big) |\cE_{N_\ell,j}(\alpha)|^2 \ll N_\ell^{2\epsilon}.
\end{equation}
\end{cor}

\subsection{Step 3} 
Let 
\begin{multline*}
 R_{\off}(N)(\alpha) := 
\frac{2|c_1|^2 \alpha^{\Theta}}{N^2} \sum_{N^{1-\epsilon} \ll j\ll N^{1+\epsilon}}  \^f\Big(\frac jN\Big) j^\Theta  
\\
 \sum_{\substack{m\neq n\\    m,n   \asymp j/N^{1-\theta}}} 
  \frac{1}{(mn)^{\frac{\Theta+1}2}}  h\Big(\frac{  (\theta\alpha j)^{\Theta}}{Nm^{\Theta}}\Big)   h\Big(\frac{  (\theta\alpha j)^{\Theta}}{Nn^{\Theta}}\Big) 
   e\left(c_2 \cdot  (\alpha j)^\Theta \left(\frac 1{m^{\Theta-1}}-\frac 1{n^{\Theta-1}}   \right) \right)
 \end{multline*}
the sum is over   $m\neq n$. 
We will show 
\begin{prop}\label{prop:R vs Roff}
For \underline{almost all} $\alpha\in [1,2]$ and all $N_\ell \in \mathcal N$ we have that
\begin{multline*}
 R_{f,h}(N_\ell)(\alpha) = \intinf f(t)dt \bigg(\intinf h(s)ds \bigg)^2  
+ R_{\off}(N_\ell)(\alpha)   +O(N_\ell^{-\epsilon/2}) .
  \end{multline*}
\end{prop}

\subsection{Step 4}
We  average over $\alpha$ and want to show that the \underline{second} 
moment of $R_{\off}(N)(\alpha)$ is small.

\begin{thm}\label{thm:2nd moment of Roff}
For any fixed $\delta \in (0,1-\theta)$, we have
\begin{equation*}
 \intinf |R_{\off}(N)(\alpha)|^2 d\mu(\alpha) \ll N^{-\delta}.
\end{equation*}
\end{thm}
By using Chebychev's inequality and the Borell--Cantelli lemma, 
we deduce from
Theorem~\ref{thm:2nd moment of Roff} 
that
$$
 \lim_{\ell \to \infty}R_{\off}(N_\ell)(\alpha)= 0
$$ 
holds for almost all $\alpha \in [1,2]$ and all $N_\ell \in \mathcal N$.
  Inserting   into Proposition~\ref{prop:R vs Roff}, we obtain
\begin{thm}\label{thm: subsequence convergence theorem}
For almost all $\alpha \in [1,2]$, and all $N_\ell \in \mathcal N$, we have
\[
 \lim_{\ell \to \infty}  R_{f,h}(N_\ell)(\alpha)  =\intinf f(t)dt \Big(\intinf h(s)ds\Big)^2.
 \]
\end{thm}
By employing a standard (deterministic)
argument,  
one deduces Theorem \ref{main thm} 
from Theorem \ref{thm: subsequence convergence theorem}.
 For instance, one can use \cite[Lemma 3.1]{RT} which give that if 
Theorem \ref{thm: subsequence convergence theorem} holds
for all $f\in C_c^\infty(\R)$,  for a strictly increasing sequence of positive integers
 $\mathcal N=\{ N_\ell \}_{\ell=1}^\infty $, with 
$
  N_{\ell+1}/N_\ell\sim  1
 $  
  then we can pass from the subsequence $\mathcal N$ 
to the set of all integers, in the sense that 
  Theorem \ref{thm: subsequence convergence theorem} 
   holds for all integers $N\geq 1$.

 \section{Applying Poisson summation}
 \subsection{Applying Poisson summation for the first time}
 The Fourier expansion of $F_N(x)$ can be written, upon 
recalling that $f$ is even (and therefore also $\^f$ is even), as
 \[
 F_N(x)= \frac 1N\sum_{j \in \Z} \^f\Big(\frac jN\Big)e(jx) =  \frac 1N\intinf f(x)dx + \frac{2}{N}\sum_{j=1}^\infty \^f\Big(\frac jN\Big)e(jx).
 \] 
Inserting into the definition of $R_{f,h}(N)(\alpha)$ we obtain
$$
R_{f,h}(N)(\alpha)  =
\frac 1{N^2}\sum_{j \in \Z} \^f\Big(\frac jN\Big) \sum_{x\neq y\geq 1}  
h\Big(\frac xN\Big)h\Big(\frac yN\Big) e(\alpha j x^\theta ) e(-\alpha j y^\theta ).
$$
Next by adding and subtracting the $x=y$ diagonal, the right hand side is
$$
\frac 1{N^2}\sum_{j \in \Z} \^f\Big(\frac jN\Big) 
\left( \sum_{x \geq 1} h\Big(\frac xN\Big) e(\alpha j x^\theta )\sum_{y \geq 1} h\Big(\frac yN\Big) e(-\alpha j y^\theta )-\sum_{x \geq 1} h\Big(\frac xN\Big)^2  \right).
$$
By recalling the definition of $\cE_{N,j}(\alpha)$, see \eqref{def: S,E}, we find that
\begin{equation}\label{convert E}
R_{f,h}(N)(\alpha) =\frac 1{N^2}\sum_{j \in \Z} \^f\Big(\frac jN\Big) |\cE_{N,j}(\alpha)|^2 - \frac 1{N^2}\sum_{j \in \Z} \^f\Big(\frac jN\Big) \sum_{x \geq 1} h\Big(\frac xN\Big)^2.
\end{equation}

Using the trivial bound $|\cE_{N,j}(\alpha)|\ll N$, we can truncate the frequencies 
$j>N^{1+\epsilon}$ with negligible error. 
Further, the term with $j=0$ contributes
\begin{equation*}
\^f(0) \frac {| \cE_{N,0}(\alpha)|^2}{N^2} = \^f(0) 
\Big|\frac 1N \sum_{y \geq 1} 
h\Big(\frac yN\Big)\Big|^2 =\intinf f(x)dx \bigg(\intinf h(y)dy\bigg)^2 +O(N^{-100})
\end{equation*}
which is the limit that we are aiming for.

The diagonal $x=y$ term equals: 
\begin{align*}
 \frac 1{N^2}\sum_{j \in \Z} \^f\Big(\frac jN\Big) \sum_{x \geq 1} h\Big(\frac xN\Big)^2 
& = \intinf \^f(x)dx \intinf h(  y)^2 dy +O(N^{-100}) \\
& =f(0)\intinf h(  y)^2 dy +O(N^{-100}) .
\end{align*}

We obtain 
\begin{align}\label{relating R and S}
R_{f,h}(N)(\alpha) =S(f,h;N)(\alpha) & + \intinf f(x)dx \bigg(\intinf h(y)dy\bigg)^2
\\ \nonumber
 & -f(0)\intinf h(y)^2 dy +O(N^{-100})    
\end{align}
where 
\begin{equation}\label{convert S to E}
S(f,h;N)(\alpha) =\frac {2}{N^2}\sum_{1\leq j \ll N^{1+\epsilon}} \^f\Big(\frac jN\Big) |\cE_{N,j}(\alpha)|^2 .
\end{equation}

\subsection{Applying van der Corput's B-process}\label{sec:stationary phase general}

To deal with the terms including the smooth exponential sum $\cE_{N,j}(\alpha) $, we  apply Poisson summation  and a stationary phase expansion (van der Corput's ``B-process'').  
Recall 
$$
\tilde \cE_{N,j}(\alpha)= c_1 \cdot  (\alpha j)^{\frac \Theta 2} 
\sum_{m   \asymp j/N^{1-\theta}} 
  \frac{1}{ m^{\frac{\Theta+1}2}}  
h\Big(\frac{(\theta\alpha j)^{\Theta}}{Nm^\Theta}\Big)   
e\Big(c_2 \frac{(\alpha j)^{\Theta}}{m^{\Theta-1}} \Big),
$$
as well as the constants $c_1$, and $c_2$ from Proposition \ref{prop: pruning}.

\begin{prop}\label{prop:apply poisson to exp sum a}
For $\alpha \in [1,2]$, and $j>0$,   
\begin{equation*} 
\cE_{N,j}(\alpha)  = \tilde \cE_{N,j}(\alpha)
 +O\Big(\frac{N^{1-  \frac{\theta}2}}{j^{1/2}}\Big).
\end{equation*}
\end{prop} 
Note that  for $1\leq j \ll N^{1-\theta}$, we just obtain an upper bound since in that case $\tilde \cE_{N,j}(\alpha)=0 $.

For proving the above proposition, we
quote the following version of the smooth B-process (see  \cite[eq 8.47]{IK}):
\begin{thm}
Let $\phi\in C^4[A,B]$  be real valued so that there are $\Lambda>0$ and $\eta\geq 1$ with
\[
\Lambda\leq |\phi^{(2)}(x)| \leq \eta \Lambda, \qquad   
|\phi^{(3)}(x)| <\frac {\eta \Lambda}{B-A}, \qquad 
|\phi^{(4)}(x)| <\frac {\eta \Lambda}{(B-A)^2}
\]
for all $x\in [A,B]$. Further, assume that $\phi^{(2)}<0$ on $[A,B]$. 
Let $a=\phi'(A)$, and $b=\phi'(B)$. Then for all smooth functions $g$, 
\begin{multline*}
\sum_{n\in [A,B]} g(n) e(\phi(n)) = \sum_{m\in [b,a]} \frac{ g(x_m)}{|\phi''(x_m)|^{1/2}} 
e\Big(\phi(x_m)-mx_m-\frac 18\Big) 
\\
+O\left(  G\Lambda^{-1/2} + G\eta^2 \log(a-b+1) \right)
\end{multline*}
where $x_m$ is the unique solution to $\phi'(x_m)=m$, and $G=|g(B)|+\int_A^B |g'(t)|dt$. 
\end{thm}

\begin{proof}[Proof of Proposition \ref{prop:apply poisson to exp sum a}]
 We apply this to the sum 
 \[
 \cE_{N,j}(\alpha)   = \sum_{y\geq 1} h\Big(\frac yN\Big)e(  \alpha j y^\theta)
 \]
 where  $g(y) = h(y/N)$, so that $[A,B]=[N,2N]$,  $G=O(1)$, and  
$ \phi(x)  =   \alpha j x^\theta$ for which
\[
\phi'(x) =  \frac{\theta \alpha j}{x^{1-\theta}}, \qquad \phi''(x) =- \frac{\theta(1-\theta)\alpha j}{x^{2-\theta}} .
\]
Therefore we can choose
\[
 \Lambda\asymp \frac{j}{N^{2-\theta}}, \quad \eta \ll 1, 
 \]
 and 
 \[
 [b,a]  \approx   \frac{j}{N^{1-\theta}} [1, 2^{1-\theta}]   
  \]
(also recall that we assume that  $\alpha \in  \supp \rho = [1,2]$). 
The critical points are solutions of
\[
 \frac{\theta \alpha j}{x_m^{1-\theta}} = m \quad \leftrightarrow \quad x_m  = \Big(\frac {\theta\alpha j}{m}\Big)^\Theta .
\]

We find
\begin{multline*}
 \cE_{N,j}(\alpha)   = \frac{\theta^{\Theta/2}}{\sqrt{1-\theta}}
 (\alpha j)^{\Theta/2} e\Big(-\frac 18\Big) \sum_{m\asymp j/N^{1-\theta}} 
\frac 1{m^{\frac{\Theta+1}{2}}} 
h\Big(\frac{(\theta \alpha j)^\Theta}{Nm^\Theta}\Big) 
e\Big(\frac{(\theta^{\Theta-1}-\theta^\Theta)(\alpha j)^\Theta}{m^{\Theta-1}}\Big) 
\\ + O\left(  \frac{N^{1-\frac{ \theta}{2}}}{j^{1/2}}  + \log N  \right)
 =\tilde \cE_{N,j}(\alpha) + O\left(  \frac{N^{1-\frac{ \theta}{2}}}{j^{1/2}} \right)
\end{multline*}
as claimed.

In particular there are no critical points unless $j \gg N^{1-\theta}$,  in which case we just obtain an upper bound. 
\end{proof}

 \section{The second moment of \ensuremath{\tilde \cE_{N,j}} and proof of Corollary~\ref{cor: truncate j less than Neps}}
Let $\rho\in C^\infty([1,2])$, $\rho\geq 0$, normalized so that $\int_1^2\rho(x)\frac{dx}{x}=1$. 
We define a smooth   measure  on $[1,2]$ by 
 \[
 d\mu(\alpha) = \frac{\Theta\rho(\alpha^\Theta) d\alpha}{\alpha}
 \]
 which satisfies $\intinf d\mu(\alpha)=1$.
Further, we require the following application of the non-stationary phase principle
dealing with the oscillatory integrals 
 \begin{multline*}
\mathcal I(j,m,n,N) :=\int_1^2 \alpha^\Theta  h\Big(\frac{(\theta \alpha j)^\Theta}{Nm^\Theta}\Big)  
h\Big(\frac{(\theta \alpha j)^\Theta}{Nn^\Theta}\Big)  
e\left(c_2
\left(\frac { (\alpha j)^\Theta}{m^{\Theta-1}} -\frac { (\alpha j)^\Theta}{n^{\Theta-1}} \right) \right) 
d\mu(\alpha) .
 \end{multline*} 
 To bound the off-diagonal terms, we invoke 
 \begin{lem}\label{lem:osc integral 1}
For all $K\geq 1$, there is some $C_K=C_{K,h,\rho}>0$ so that for all distinct  
$m,n\asymp j/N^{1-\theta}$, and all $j\gg N^{1-\theta}$, we have that
\begin{equation*}
\left | \cI(j,m,n,N) \right| 
\leq C_K N^{-K} .
\end{equation*}
\end{lem}

\begin{proof}

Changing variables $\beta = \alpha^\Theta$, and recalling the definition \eqref{def of dmu} of
\[
d\mu(\alpha) =\frac{ \Theta \rho(\alpha^\Theta)d\alpha}{\alpha }  = \frac{\rho(\beta)d\beta}{\beta},
\]  
the oscillatory integral $\cI(j,m,n,N) $  can be written as
\begin{equation*}
  \mathcal I(j,m,n,N) =\int_0^\infty  e\left( \beta  \cdot c_2 j^\Theta\left(\frac 1{m^{\Theta-1}} -\frac 1{n^{\Theta-1}} \right) \right)  
   h(\beta  x_m)   h(\beta x_n)    
 \rho(\beta) d\beta
\end{equation*}
with
\[
 x_m = \left( \frac{\theta j}{N^{1-\theta}m} \right)^\Theta .
\]
Since $\rho$ and $h$ are supported in the interval $[1,2]$, we must have $\beta\in [1,2]$ and $\beta x_m, \beta x_n  \in [1,2]$. Hence for the integral  to be  nonzero, we must have 
\[
x_m,x_n \in \Big[\frac 12,2\Big].  
\]

The integral is the Fourier transform of the function $F\in C_c^\infty(\R)$ (which depends on $j,m,n,N$):
\[
F(\beta) = \rho(\beta)   h(\beta x_m)h(\beta x_n) , 
\]
evaluated at the point $-c_2 j^\Theta\left(\frac 1{m^{\Theta-1}} -\frac 1{n^{\Theta-1}} \right)$. For any $F\in C_c^\infty(\R)$, we can bound the Fourier transform using integration  by parts $K$ times:
 \[
 |\^F(y)|\leq \frac 1{(2\pi |y|)^K}\intinf |F^{(K)}(\beta)|d\beta, \quad y\neq 0 .
 \] 
 In our case, $|F|_\infty \ll 1$, and the derivatives of $F$ are bounded by 
 \[
 |F^{(K)}|_\infty \ll C_{h,\rho}\left(1+\max \left( x_m^K, x_n^K\right)\right)
 \]
 and since $x_m,x_n = O(1)$ we obtain $  |F^{(K)}|_\infty =O(1)$.
 Therefore if $m\neq n$ then  
 \[
  |\^F(y)| \ll_K \frac 1{|y|^K} .
  \] 

Finally, note that if $m\neq n$  (but both $m,n\asymp j/N^{1-\theta}$), 
then the frequency $y$ of the Fourier transform is bounded below by 
\begin{align*}
|y| & =  \bigg | c_2 j^\Theta\left(\frac 1{m^{\Theta-1}} -\frac 1{n^{\Theta-1}}\right) \bigg|\\
& \gg j^\Theta\left|\frac 1{n^{\Theta-1}} -\frac 1{m^{\Theta-1}} \right| 
\\ 
& \gg \frac{j^\Theta}{ (mn)^{\Theta-1}}(n-m) m^{\Theta-2} \\
& \gg \frac{j^\Theta N^{2(\Theta-1)(1-\theta)}}{ j^{2(\Theta-1)}} 
\cdot 1 \cdot
 (\frac{j}{N^{1-\theta}})^{\Theta-2} = N^{\Theta(1-\theta)} = N.
 \end{align*}
Thus $| \mathcal I(j,m,n,N)|\ll |\^F(y)|\ll_K N^{-K}$. 
\end{proof}
We proceed to show that $| \tilde \cE_{N,j}(\alpha) |^2$
exhibits, essentially, square-root cancellation on average over $\alpha \in [1,2]$. 
An important step to this end is:
\begin{prop}\label{prop second moment of E} 
Assume $N^{1-\theta} \ll j<N^2$. Then 
 \[
 \intinf | \tilde \cE_{N,j}(\alpha) |^2 d\mu(\alpha) \ll N .
  \]
 \end{prop}
\begin{proof}
Recall that 
\[
 \tilde\cE_{N,j}(\alpha)=c_1 (\alpha j)^{\Theta/2} 
\sum_{m\asymp j/N^{1-\theta}} \frac 1{m^{\frac{\Theta+1}{2}}} 
h\Big(\frac{(\theta \alpha j)^\Theta}{Nm^\Theta}\Big) 
e\Big(\frac{c_2 (\alpha j)^\Theta}{m^{\Theta-1}}\Big).
 \]
 Hence
\begin{equation}\label{bound for E via I}
\begin{split}
\int_1^2 |\tilde\cE_{N,j}(\alpha)|^2 d\mu(\alpha) & \ll
j^\Theta \sum_{m,n\asymp j/N^{1-\theta}}
\frac 1{(mn)^{\frac{\Theta+1}{2}}} |\mathcal I(j,m,n,N)|
\\
&\ll \frac{N^{2-\theta}}{j} \sum_{m,n\asymp j/N^{1-\theta}} |\mathcal I(j,m,n,N)| .
\end{split}
\end{equation}
 
The diagonal terms $m=n$ contribute
\[
 \ll \frac{N^{2-\theta}}{j} \sum_{m\asymp j/N^{1-\theta}} 1
\ll \frac{N^{2-\theta}}{j}\frac{j}{N^{1-\theta}} = N .
\]

We  conclude the proof of Proposition~\ref{prop second moment of E} by deducing that  the total contribution 
of all the off-diagonal terms to \eqref{bound for E via I} is bounded by 
\begin{multline*}
\frac{N^{2-\theta}}{j} 
\sum_{\substack{m,n \asymp j/N^{1-\theta}\\ m\neq n}} |\mathcal \cI(j,m,n,N)| 
\ll \frac{N^{2-\theta}}{j}\Big(\frac{j}{N^{1-\theta}}\Big)^2 N^{-100} 
\ll jN^{-99} \ll N^{-90}
 \end{multline*}
since we  assumed that $j\ll N^2$.
\end{proof}
We can now proceed to the:
\begin{proof}[Proof of Corollary~\ref{cor: truncate j less than Neps}]
Recall $N_\ell = \lfloor \ell^{C}\rfloor $ as well as, see \eqref{convert S to E}, that 
\begin{equation*} 
S(f,h;N)(\alpha) =\frac {2}{N^2}
\sum_{1\leq j\ll N^{1+\epsilon} } \^f\Big(\frac jN\Big) |\cE_{N,j}(\alpha)|^2.
\end{equation*}
Now fix $\tau\in \{1-\epsilon, 1+ \epsilon \}$. We observe,
by using Proposition~\ref{prop:apply poisson to exp sum a}, that
\begin{align}
\frac {2}{N^2}
\bigg| \sum_{1\leq j\ll N^\tau} \^f\Big(\frac jN\Big) |\cE_{N,j}(\alpha)|^2 
\bigg| 
& \ll 
\frac 1{N^2} \sum_{1\leq j\ll N^\tau } 
\left( |\tilde{\cE}_{N,j}(\alpha)|^2 + \frac{N^{2-\theta}}{j} \right) \nonumber
 \\
& \ll 
\frac 1{N^2} \bigg( \sum_{1\leq j\ll N^\tau } 
 |\tilde{\cE}_{N,j}(\alpha)|^2 \bigg) + \frac{1}{N^{\epsilon}}
\label{eq: intermediate crude bound}
\end{align}
Next, we consider 
$$
Y_{N,\tau}(\alpha) := \frac{1}{N^{1+\tau}}
\sum_{j\ll N^{\tau}} \vert \tilde \cE_{N,j}(\alpha)\vert^2,
$$ 
and infer from Proposition \ref{prop second moment of E} that
$$
\intinf Y_{N,\tau}(\alpha)  \, \mathrm{d}\mu(\alpha)
 \ll 1.
$$
By the Borel--Cantelli lemma and Markov's inequality,
$$ 
Y_{N_\ell,\tau}(\alpha) \ll N_\ell^{\epsilon}
$$  
almost all $\alpha\in [1,2]$.
Using this bound in \eqref{eq: intermediate crude bound} completes the proof.
\end{proof}

\section{Proof of Proposition~\ref{prop:R vs Roff}  }

\begin{proof}
Recall \eqref{relating R and S}:
\begin{multline}\label{relating R and S v2}
R_{f,h}(N)(\alpha) =S(f,h;N)(\alpha) + \intinf f(x)dx \bigg(\intinf h(y)dy\bigg)^2
\\
 -f(0)\intinf h(  y)^2 dy +O(N^{-100})    
\end{multline}
with
\begin{equation}\label{convert S to E v2}
S(f,h;N)(\alpha) =\frac {2}{N^2}\sum_{1\leq j \ll N^{1+\epsilon}} \^f\Big(\frac jN\Big) |\cE_{N,j}(\alpha)|^2 .
\end{equation}
Inserting Proposition~\ref{prop:apply poisson to exp sum a} 
and Corollary~\ref{cor: truncate j less than Neps}
in \eqref{convert S to E v2} gives 
\begin{multline*}
S(f , h;N)(\alpha) =\frac {2}{N^2}\sum_{N^{1-\epsilon} \leq j\ll N^{1+\epsilon}} \^f\Big(\frac jN\Big) |\cE_{N,j}(\alpha)|^2 +O(N^{-\epsilon/2}) 
\\
=  
\frac {2}{N^2}\sum_{1\leq j\ll N^{1+\epsilon}} \^f\Big(\frac jN\Big) |\tilde \cE_{N,j}(\alpha)|^2 
\\
 +O\left(   \frac{N^{1-\theta/2}}{N^2} \sum_{1\leq j\ll N^{1+\epsilon}} 
 \frac{\vert \^f(\frac jN)\vert }{j^{1/2}} |\tilde \cE_{N,j}(\alpha)|
 +N^{-\theta}\sum_{1\leq j\ll N^{1+\epsilon}} \frac{\vert \^f(\frac jN)\vert }{j}  \right) .
   \end{multline*}
 
The second term in the $O$-symbol can be dispensed with by
  \[
 \sum_{1\leq j\ll N^{1+\epsilon}} \frac{\vert \^f(\frac jN)\vert }{j}
\ll  \sum_{1\leq j \ll N^{1+\epsilon}} \frac 1j      \ll   \log N .
 \]
We proceed to bound the first term in the $O$-symbol.
By using the estimate just above and $\vert \^f(\frac jN) \vert = O(1)$,
the Cauchy--Schwarz inequality implies
\begin{multline*}
     \frac {N^{1-\theta/2} }{N^2} \sum_{1\leq j\ll N^{1+\epsilon}} 
     \frac{\vert \^f(\frac jN)\vert}{j^{1/2}} |\tilde \cE_{N,j}(\alpha)| 
    \ll   N^{-1/2-\theta/2+o(1)} 
\Bigg(\sum_{1\leq j\ll N^{1+\epsilon}}  
|\tilde \cE_{N,j}(\alpha)|^2\Bigg)^\frac{1}{2}.
  \end{multline*}
Hence we infer from \eqref{eq: crude upper bound}
that for almost all $\alpha \in [1,2]$ the bound
   $$
     \frac {N_\ell^{1-\theta/2} }{N_\ell^2} \sum_{1\leq j\ll N_\ell^{1+\epsilon}} 
     \frac{\vert \^f(\frac {j}{N_\ell})\vert}{j^{1/2}} |\tilde \cE_{N_\ell,j}(\alpha)| 
 \ll  N_\ell^{ -\theta/2+2\epsilon}
  $$
holds. The upshot is that, for almost all $\alpha$
and all $N=N_\ell\in \mathcal N$, we have
\[
S(f , h;N_\ell)(\alpha) =\frac {2}{N_\ell^2}
\sum_{N_\ell^{1-\epsilon} \leq j\ll N_\ell^{1+\epsilon}} 
\^f\Big(\frac {j}{N_\ell}\Big) |\tilde \cE_{N_\ell,j}(\alpha)|^2 
+O(N_\ell^{-\theta/2+O(\epsilon)}) .
\]

We for ease of notation, we write now $N$ in place of $N_\ell$.
Thus we find that for almost all $\alpha$,  and all $N\in \mathcal N$,
\begin{equation}\label{S in terms of tilde S}
 S(f,h;N)(\alpha)=\tilde S(f,h;N)(\alpha) +O(N^{-\epsilon/2})
\end{equation}
 where
\[
\tilde S(f,h;N)(\alpha) = \frac {2}{N^2}\sum_{N^{1-\epsilon} \ll j \ll N^{1+\epsilon}  }  \^f\Big(\frac jN\Big) |\tilde \cE_{N,j}(\alpha)|^2 .
\]

We write out
\begin{multline}\label{second exp for R gen}
\tilde S(f , h;N)(\alpha) =\frac {2}{N^2}\sum_{N^{1-\epsilon} \ll j\ll N^{1+\epsilon}  } \^f\Big(\frac jN\Big) |\tilde \cE_{N,j}(\alpha)|^2 
\\= 
\frac{2|c_1|^2 \alpha^{\Theta}}{N^2} \sum_{N^{1-\epsilon} \ll j \ll  N^{1+\epsilon}  }  \^f\Big(\frac jN\Big) j^\Theta  
\\
 \sum_{m,n \asymp j/N^{1-\theta}} 
  \frac{1}{(mn)^{\frac{\Theta+1}2}}  h\Big(\frac{  (\theta\alpha j)^{\Theta}}{Nm^{\Theta}}\Big)   h\Big(\frac{  (\theta\alpha j)^{\Theta}}{Nn^{\Theta}}\Big) 
   e\left(c_2\; (\alpha j)^\Theta \left(\frac 1{m^{\Theta-1}}-\frac 1{n^{\Theta-1}}   \right)\right) .
   \end{multline}
Restricting the sum over $m,n$ to the diagonal $m=n$ gives a term of the form 
\[
\sum_{n\geq 1} \frac 1{n^{\Theta+1}} h\Big( \frac{W}{n^\Theta }\Big)^2\ 
= \frac {1-\theta}{ W} \intinf h(y)^2dy +O(W^{-100}), 
\]
with
\[ 
W=\frac{  (\theta\alpha j)^{\Theta}}{N} \gg N^{\Theta-1-\epsilon} \to \infty .
\]
Thus the contribution of the diagonal $m=n$ to \eqref{second exp for R gen} is
\begin{equation*}
\begin{split}
\mbox{terms with }(m=n)& +|c_1|^2   \frac{1-\theta}{\theta^\Theta} \frac 1N   \sum_{j \geq 1}  \^f\Big(\frac jN\Big) \intinf h(y)^2dy +O(N^{-10})
\\
& =  f(0)\intinf h(y)^2dy  +O(N^{-10})
\end{split}
\end{equation*}
which exactly cancels off the diagonal term in \eqref{convert E}.  Thus 
\begin{equation}\label{tilde S in terms of Roff} 
\tilde S(f , h;N)(\alpha) =  f(0)\intinf h(y)^2dy   + R_{\off}(N)(\alpha)  +O(N^{-10})
\end{equation}
where the off-diagonal terms are $ R_{\off}(N)(\alpha)$, that is
\begin{multline}\label{third exp for R}
\frac{2|c_1|^2 \alpha^{\Theta}}{N^2} \sum_{N^{1-\epsilon} \ll j\ll N^{1+\epsilon}  } \^f\Big(\frac jN\Big) j^\Theta  
\\
 \sum_{\substack{m\neq n\\    m,n   \asymp j/N^{1-\theta}}} 
  \frac{1}{(mn)^{{\frac{\Theta+1}2}}}  h\Big(\frac{  (\theta\alpha j)^{\Theta}}{Nm^{\Theta}}\Big)   h\Big(\frac{  (\theta\alpha j)^{\Theta}}{Nn^{\Theta}}\Big) 
   e\Big(c_2\; (\alpha j)^\Theta\Big(\frac 1{m^{\Theta-1}}-\frac 1{n^{\Theta-1}}  \Big)\Big).
 \end{multline}

Inserting \eqref{S in terms of tilde S} and \eqref{tilde S in terms of Roff}  into \eqref{relating R and S v2} we obtain
 \begin{multline*}
 R_{f,h}(N)(\alpha) = \intinf f(t)dt (\intinf h(s)ds)^2  + R_{\off}(N)(\alpha)    
+O(N^{-\epsilon/2})
 \\
  \end{multline*}
 for a.e. $\alpha\in [1,2]$  and all $N\in \mathcal N$,
 which is Proposition~\ref{prop:R vs Roff}.
 \end{proof}

\section{The second moment of \ensuremath{R_{\off}(N)(\alpha)}}
We now average $|R_{\off}(N)(\alpha) |^2$ over $\alpha$, 
by using the measure \eqref{def of dmu}. 
The goal is to show Theorem~\ref{thm:2nd moment of Roff}, namely
\begin{equation}\label{variance bound} 
\intinf  \left| R_{\off}(N)(\alpha) \right|^2  d\mu(\alpha) \ll N^{-\delta}
\end{equation}
for $\delta \in (0,1-\theta)$.
To this end, we require a simple lemma on oscillatory integrals.
By arguing as in the proof of Lemma~\ref{lem:osc integral 1} we find:
\begin{lem}\label{lem:osc integral}
With $z_1= m_1^{1-\Theta}- n_2^{1-\Theta}$ and  
$z_2= m_2^{1-\Theta}- n_2^{1-\Theta}$, we define
\begin{multline*}
 \mathcal I(\vec j, \vec m,\vec n, N) = 
 \intinf e\left(c_2 \cdot  \alpha^\Theta \left( j_1^\Theta z_1 -j_2^\Theta z_2   \right) \right)
   \prod_{r=1}^2  h\Big(\frac{  (\theta\alpha j)^{\Theta}}{Nm_r^{\Theta}}\Big)   
h\Big(\frac{  (\theta\alpha j)^{\Theta}}{Nn_r^{\Theta}}\Big)   
\alpha^{2\Theta} d\mu(\alpha).
\end{multline*}
For all $K\geq 1$, there is some $C=C_{K,h,\rho}$ so that for all  $m,n,j$, 
\begin{equation*}
\left |\mathcal I(\vec j, \vec m,\vec n, N)  \right| 
\leq C \min \left( 1, \frac 1{ |j_1^\Theta z_1 -j_2^\Theta z_2|^K} \right)  .
\end{equation*}
Moreover, $\mathcal I(\vec j, \vec m,\vec n, N)=0$ unless
\begin{equation}\label{condition on m vs j}
m_r, n_r \asymp j_r/N^{1-\theta} \quad(r=1,2).
\end{equation}
\end{lem}
Next, we relate the variance of $R_{\off}$ 
to counting the solutions of Diophantine inequalities:
\begin{prop}\label{prop:var}
Let $U=(1+\epsilon) \log N$, $Q=(\theta+\epsilon) \log N$. Then
\begin{equation}\label{eq: second moment of Roff}
\intinf  \left| R_{\off}(N)(\alpha) \right|^2  d\mu(\alpha) \ll \frac {(\log N)^2}{N^{2\theta} } \sum_{(1-\epsilon)\log N<u\leq U}\sum_{0\leq q\leq Q} e^{-2u} \mathcal D_{u,q}
\end{equation}
where 
\begin{equation}\label{def of Du}
\mathcal D_{u,q} = \#\left\{ 
\begin{array}{l}
 e^u\leq  j_1,j_2<e^{u+1},
 m_1,m_2,n_1,n_2\asymp \frac{e^u}{N^{1-\theta}},
\\ \\  
e^q \leq n_1-m_1, n_2-m_2 < e^{q+1}:
\\  \\
\left| j_1^\Theta(\frac 1{m_1^{\Theta-1}}-\frac 1{n_1^{\Theta-1}}) -j_2^\Theta (\frac 1{m_2^{\Theta-1}}- \frac 1{n_2^{\Theta-1}})  \right| \leq N^\epsilon
\end{array}
\right\}.
\end{equation}
\end{prop}

\begin{proof}
We rewrite \eqref{third exp for R}  as 
\begin{multline*}
 R_{\off}(N)(\alpha)   = \frac { 2|c_1|^2 \alpha^\Theta  }{N^2}\sum_{(1-\epsilon)\log N \leq u\leq U}\sum_{0\leq q\leq Q} 
 \sum_{e^u \leq j<e^{u+1}} \^f\Big(\frac jN\Big) j^\Theta  
 \\
  \sum_{\substack{m, n \asymp e^u/N^{1-\theta}\\e^q\leq  n-m<e^{q+1}}}  
\frac 1{(mn)^{\frac{\Theta+1}{2}}} 
h\Big(\frac{  (\theta\alpha j)^{\Theta}}{Nm^{\Theta}}\Big)   
h\Big(\frac{  (\theta\alpha j)^{\Theta}}{Nn^{\Theta}}\Big) 
   e\left(c_2 \cdot  (\alpha j)^\Theta 
\left(\frac 1{m^{\Theta-1}}-\frac 1{n^{\Theta-1}}   \right) \right).
\end{multline*}  
Integrating over $\alpha$ and applying Cauchy-Schwarz gives
that the left hand side of \eqref{eq: second moment of Roff} is 
at most a constant times
\begin{multline*}
\frac {UQ}{N^4}\sum_{\substack{(1-\epsilon)\log N \leq u\leq U\\ 0\leq q \leq Q}} 
 j_1^\Theta j_2^\Theta 
\sum_{\substack{e^u\leq j_1, j_2 <e^{u+1}  \\ 
m_1, n_1, m_2, n_2 \asymp e^u/N^{1-\theta}\\
e^q\leq n_i-m_i<e^{q+1}}} 
\frac {| \mathcal I(\vec j, \vec m,\vec n, N)|}{(m_1n_1m_2n_2)^{\frac{1+\Theta}{2}}}.
\end{multline*}

We  see from Lemma~\ref{lem:osc integral} that the oscillatory integral $\mathcal I(\vec j, \vec m,\vec n, N)$ vanishes unless \eqref{condition on m vs j} holds, 
and is negligible unless 
\begin{equation}
\left| j_1^\Theta z_1-j_2^\Theta z_2 \right| \leq  N^\epsilon.
\end{equation}
Hence, 
noting $j_r^\Theta/(m_rn_r)^{\frac {1+\Theta}{2}}\asymp e^{-u}N^{2-\theta}$, 
we infer $\intinf  \left| R_{\off}(N)(\alpha) \right|^2  d\mu(\alpha)$
is at most a constant times
\begin{align*}
& \frac{(\log N)^2}{N^4} 
\sum_{\substack{(1-\epsilon)\log N\leq u\leq U\\ 0\leq q\leq Q}}   
\sum_{\substack{ e^u\leq j_1, j_2<e^{u+1} 
\\m_i,n_i\asymp e^u/N^{1-\theta} 
\\ e^q\leq n_i-m_i<e^{q+1} }}   
  e^{-2u}N^{4-2\theta}
  \mathbf{1}\left\{ \left| j_1^\Theta z_1-j_2^\Theta z_2 \right| \leq N^\epsilon \right\}
 \\
& \ll   \frac  {(\log N)^2}{N^{2\theta} } \sum_{\substack{(1-\epsilon)\log N\leq u\leq U\\ 0\leq q\leq Q}}  e^{-2u}  
   \sum_{\substack{ e^u\leq j_1, j_2<e^{u+1} \\m_i,n_i\asymp e^u/N^{1-\theta}
 \\ e^q\leq n_i-m_i<e^{q+1} }} 
   \mathbf{1}\left\{ \left| j_1^\Theta z_1-j_2^\Theta z_2 \right| \leq  N^\epsilon \right\}
\end{align*}
which is the statement of Proposition~\ref{prop:var}. 
\end{proof}

 \section{Decoupling and using the Robert-Sargos theorem}
 
  We now bound $\mathcal D_{u,q}$ by using an integral representation, a modification of that used by Aistleitner, El-Baz, and Munsch  \cite{AEM}, and then invoke classical results about Dirichlet polynomials. 
  
Recall that $\mathcal D_{u,q}$  is given by \eqref{def of Du}:  
\[
\mathcal D_{u,q}:=\#\left\{e^u\leq j_1,j_2< e^{u+1},\;  z_1,z_2\in \mathcal Z_{u,q}: \quad  |j_1^\Theta z_1-j_2^\Theta z_2|< N^\epsilon \right\}
\]
where $N^{1-\epsilon}<e^u<N^{1+\epsilon}$ and
\[
  \mathcal Z_{u,q}=\left\{z=\frac 1{m^{\Theta-1}}-\frac 1{n^{\Theta-1}}: 
  m,n\asymp \frac{e^u}{N^{1-\theta}}, \quad e^q\leq n-m <e^{q+1} \right\}  
\]
(we consider it as a multi-set, that is the elements are counted with their multiplicities, if any).

Note that $ z\in \mathcal Z_{u,q}$ satisfies
$z\asymp e^qN/e^{\Theta u}$ and $1\leq e^q \leq e^u/N^{1-\theta}$. Further,
$$\# \mathcal Z_{u,q} \asymp \frac{e^u}{N^{1-\theta}}e^q .$$ 

\begin{thm}\label{thm:bound on Duq}
If $N^{1-\epsilon} \ll e^u\ll N^{1+\epsilon} $,  then 
\begin{equation*}
  \mathcal D_{u,q}  \ll N^{1+3\theta+o(1)}
\end{equation*}
\end{thm}

Together with Proposition~\ref{prop:var}, this will give the required bound \eqref{variance bound}  on the variance, 
that is prove Theorem~\ref{thm:2nd moment of Roff}.

\subsection{Estimating \ensuremath{\mathcal D_{u,q}} by an integral}
We define (following Regavim \cite{Regavim})  
\[
P_{u,q}(t):=\sum_{z\in \mathcal Z_{u,q}} z^{2\pi it} .
\]
Also let $D_u(t)$ be the Dirichlet polynomial
\[
D_u(t):= \sum_{e^u\leq j <e^{u+1}  } j^{2\pi it}, 
\]
and  set 
\[
 T =  \frac{e^{q}N}{N^\epsilon} .
\]
Note that $T\asymp j^\Theta z/N^\epsilon$ when $j\asymp e^u$ and $z\in \mathcal Z_{u,q}$, and that $T\ll N^{1+\theta}$.

The following lemma produces useful cut-off functions, by using Beurling-Selberg functions, see Appendix~\ref{app:BS function}.
\begin{lem}\label{lem: Beurling-Selberg function}
There is a smooth function $\Phi : \R \rightarrow \R$ satisfying 
\begin{enumerate}
 \item $\Phi $ is supported in $[-1,1]$,   
\item  $\Phi\geq 0$ is non-negative, 

\item \label{ft majorizes indicator}
the Fourier transform
 $\^\Phi\geq 0$ is non-negative,
 \item   $\^\Phi(x)\geq 1$ for $|x|\leq 1$. 
 \end{enumerate}
 \end{lem}
   
  We  first show:
\begin{lem}\label{lem:D in terms of integral} 
We can   bound $\mathcal D_{u,q}$ by the twisted second moment
\begin{equation}\label{eq: intermediate estimate}
\mathcal D_{u,q} \ll \frac 1T \intinf |D_u(\Theta t)|^2 |P_{u,q}(t)|^2 \Phi\Big(\frac tT\Big) dt .
\end{equation}
\end{lem}
\begin{proof}
We expand the integral
\[
\begin{split}
\intinf |D_u(\Theta t)|^2 |P_{u,q}(t)|^2 \Phi\Big(\frac tT\Big) \frac{dt}{T} & =
 \sum_{\substack{e^u \leq  j_1,j_2 < e^{u+1}\\ z_1,z_2\in \mathcal Z_{u,q}}} 
\intinf e^{-2\pi it \log \frac{j_1^\Theta z_1}{j_2^\Theta z_2}}  \Phi\Big(\frac tT\Big) 
\frac{dt}{T}
\\
& =  \sum_{\substack{e^u \leq  j_1,j_2 < e^{u+1}\\ z_1,z_2\in \mathcal Z_{u,q}}} 
\^\Phi\Big(T\log \frac{j_1^\Theta z_1}{j_2^\Theta z_2} \Big) .
\end{split}
\]
Since $\^\Phi\geq 0$, we decrease the right hand side if we drop all terms except those where $|j_1^\Theta z_1-j_2^\Theta z_2|<N^{\epsilon}$. 
For these, we have
\[
\Big| \log \frac{j_1^\Theta z_1}{j_2^\Theta z_2} \Big|=
\Big| \log \left(1 + \frac{j_1^\Theta z_1-j_2^\Theta z_2}{j_2^\Theta z_2} \right)\Big|  
\sim \frac {|j_1^\Theta z_1-j_2^\Theta z_2|} {j_2^\Theta z_2}  
 < \frac {N^{\epsilon}}{e^{\Theta u} \frac{N  e^q }{e^{\Theta u}}} =
\frac 1T
\]
and therefore if $|j_1^\Theta z_1-j_2^\Theta z_2|<N^{\epsilon}$ then 
$|T\log \frac{j_1^\Theta z_1}{j_2^\Theta z_2}|<1$ so that 
\[
 \^\Phi\Big(T\log \frac{j_1^\Theta z_1}{j_2^\Theta z_2} \Big)  \geq 1 .
\]
Therefore
\[
\frac 1T \intinf |D_u(\Theta t)|^2 |P_{u,q}(t)|^2  \Phi\Big(\frac tT\Big) dt \geq 
 \sum_{\substack{e^u \leq  j_1,j_2 < e^{u+1}\\ z_1,z_2\in \mathcal Z_{u,q} \\ |j_1^\Theta z_1-j_2^\Theta z_2|<1}} 1 =: \mathcal D_{u,q} .
\]
\end{proof}

\subsection{Using a theorem of Robert--Sargos}
The diagonal sub-system 
\begin{align*}
\mathcal Z_{u,q, {\mathrm{diag}}} = \left\{z_1,z_2\in \mathcal Z_{u,q}: \, |z_1-z_2|< N^{2\epsilon} e^{- \Theta u} \right\}
\end{align*}
of $\cZ_{u,q}$ plays an important role in the following.
Our current goal is to show that it has essentially only diagonal solutions $z_1 = z_2$.
To this end, the next result of Robert and Sargos \cite[Thm. 2]{RS}
is crucial:
\begin{thm}\label{thm: RS}
Suppose $1 - \Theta \in \mathbb{R}\setminus\{0,1\}$, and $\epsilon>0$. Then
there exists a $C_{\epsilon,\Theta}>0$ so that 
$$ 
\# \{(x_{1},y_{1},x_{2},y_{2})\in\left[1,M\right]^{4}\cap \mathbb{Z}^4:
\vert x_{1}^{1 - \Theta}-y_{1}^{1 - \Theta}
+x_{2}^{1 - \Theta}-y_{2}^{1 - \Theta} \vert 
<\gamma M^{1 - \Theta}
\} 
$$
is at most
$
C_{\epsilon,\Theta}
M^{\epsilon}(M^{2}+\gamma M^{4})
$ for any $\gamma>0$ and any $M\geq1$.
\end{thm}
Now we are ready to bound $\# \mathcal Z_{u,q, {\mathrm{diag}}}$:
\begin{lem}\label{lem: diophantine inequality upper bound}
We have $ \# \mathcal Z_{u,q, {\mathrm{diag}}} \ll N^{2 \theta + O(\epsilon)}$.
\end{lem}
\begin{proof}
First of all note that $\mathcal Z_{u,q, {\mathrm{diag}}}$ equals
$$
\{ m_1,n_1,m_2,n_2 \ll M:  | m_1^{1-\Theta}-n_1^{1-\Theta}
- m_2^{1- \Theta}+n_2^{1-\Theta} |<  \gamma M^{(1-\Theta)} 
\}
$$
where
\[
M:=e^u/N^{1-\theta}, \quad \gamma := N^{2\epsilon}e^{-\Theta u} M^{\Theta-1} .
\]
We compute that $\gamma=  N^{2\epsilon}/(e^{u}N^{\theta})$. 
Therefore,
$$
\gamma  M^2 \ll 
N^{O(\epsilon)}
N^{2\theta} N^{-(1+\theta)}
\ll 1.
$$
Applying Theorem \ref{thm: RS} completes the proof.
\end{proof}
Now we complete the proof Theorem \ref{thm:bound on Duq} by proving:
  \begin{lem}\label{lem:P reproduces Zdiag}
The right hand side of \eqref{eq: intermediate estimate} is 
$O(e^u N^{3 \theta + O(\epsilon)})$.
  \end{lem}
  
\begin{proof}
As $ \Phi(t/T)=0$ if $|t|>T$, we first note that
$$ 
\intinf |D_u(\Theta t)|^2 |P_{u,q}(t)|^2  \Phi \Big(\frac tT\Big) dt 
= 2\int_{0}^{T} |D_u(\Theta t)|^2 |P_{u,q}(t)|^2  \Phi \Big(\frac tT\Big) dt.
$$
Denote by $I_{\mathrm{low}}$  (resp. by $I_{\mathrm{high}}$) 
the contribution of $t\in [0,e^u]$
(resp. of $t\in [e^u,T]$) to the right hand side.
To bound $I_{\mathrm{low}}$, we use the trivial bound
\[
|P_{u,q}(t)|\leq P_{u,q}(0)=\#\mathcal Z_{u,q} 
\ll N^{\theta + O(\epsilon)} e^q
\]
combined with the bound (see \cite[Eq. (34) on page 141]{Montgomery})  
\[
|D_u(t)|\ll \frac{e^u}{\sqrt{1+t^2}}
\quad \mathrm{for} \,t \in [0, e^u].
\]
Thus,
\[
 I_{\mathrm{low}}  \ll \int_{0}^{e^{u}}
\frac{e^{2u}}{1+t^2} N^{4 \theta + O(\epsilon)} dt 
\ll 
e^{2u} N^{2\theta + O(\epsilon)} e^{2q}.
\]
To estimate $I_{\mathrm{high}}$,
we use, see \cite[Eq. (34) on page 141]{Montgomery}, that
\[
|D_u(t)|\ll t^{1/2}, \quad \mathrm{for} \,\, t \in [e^u, e^{2u}];
\] 
by combining this with Lemma \ref{lem:P reproduces Zdiag},
we infer that
\[
I_{\mathrm{ high}} \ll \int_{e^{u}} ^{T}
|t^{1/2}|^2 |P_{u,q}(t)|^2 \Phi \Big(\frac tT\Big) dt
\ll T N^{2 \theta + O(\epsilon)} .
\]
All in all, upon bearing $T \gg e^q N^{1 + O(\epsilon)}$ in mind, we conclude
\[
\frac 1T \intinf |D_u(\Theta t)|^2 |P_{u,q}(t)|^2  \Phi \Big(\frac tT\Big) dt \ll 
\frac{I_{\mathrm{low}}+ I_{\mathrm{ high}}}{T}
\ll  N^{ O(\epsilon)}(e^u N^{2\theta} e^q
+ N^{2 \theta}).
\]
Using $e^q = O(N^{\theta+O(\epsilon)})$ completes the proof.
\end{proof}

\appendix
\section{  Construction of \ensuremath{\Phi}}\label{app:BS function}
Here we prove the existence of the well-behaved cut-off function $\Phi$.
\begin{proof}[Proof of Lemma \ref{lem: Beurling-Selberg function}]
 
 We start with a function $\Psi_+$, supported in $[-1,1]$, so that $\^\Psi_+\geq \mathbf 1_{[-1,1]}$ is a majorant  
   for the indicator function of the interval $[-1,1]$, meaning that $\^\Psi_+\geq 0$, and $\^\Psi_+(x)\geq 1$ for $|x|\leq 1$. Such  functions were constructed by Beurling and Selberg, cf \cite{Vaaler}, starting from Beurling's function
   \[
   B(x) =  \left( \frac{\sin \pi x}{\pi }\right)^2 \left( \frac 2x +\sum_{n=0}^\infty \frac 1{(x-n)^2} 
   -\sum_{n=1}^\infty \frac 1{(x+n)^2} \right)
   \] 
   which is a smooth majorant for the sign function, with Fourier transform $\^B$ supported on $[-1,1]$, and taking 
   \[
   \^\Psi_+(x) = \frac 12 \left( B( 1-x)+B(1+x) \right) ,
   \]
see Figure~\ref{fig:selbergmajorant}.      By construction, both  $\Psi_+$ and $\^\Psi_+$ are even and  real valued. 
       \begin{figure}[!htb]
\begin{center}
  \includegraphics[width=80mm]
{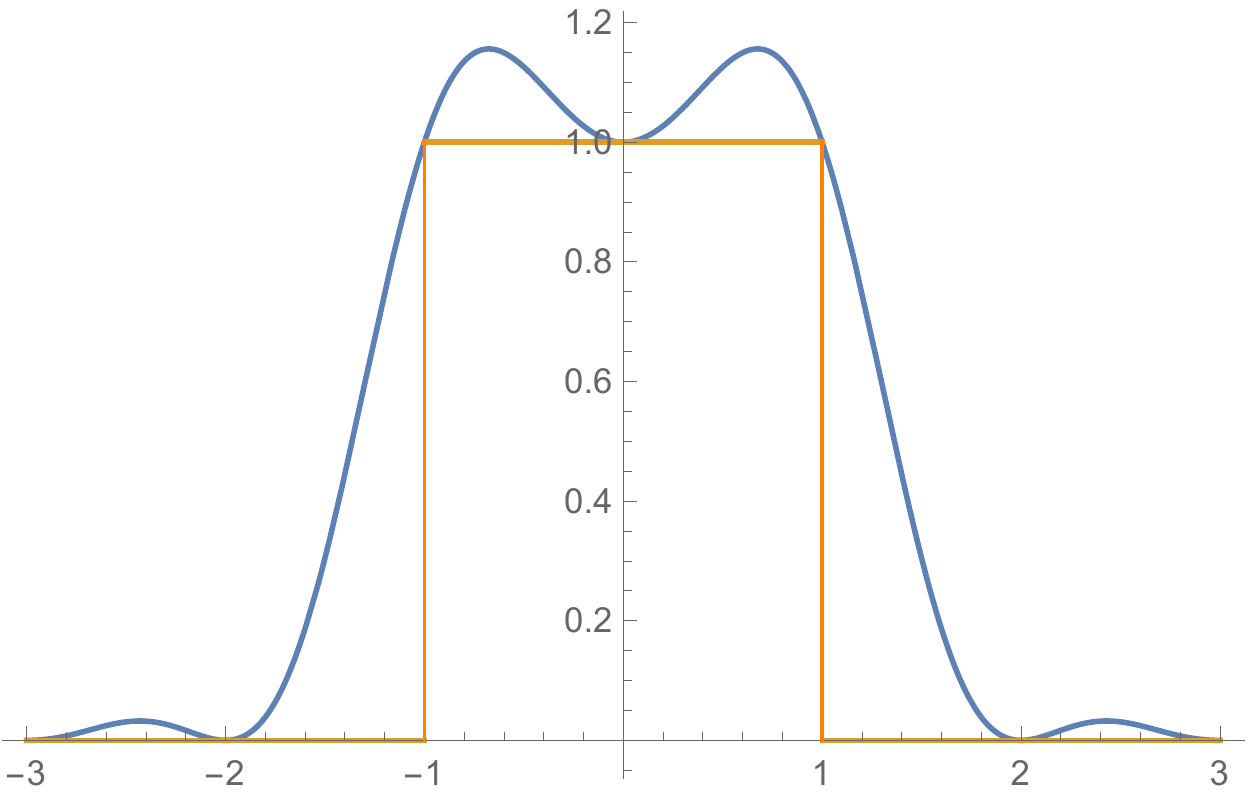}
 \caption{The Beurling-Selberg majorant \ensuremath{\^\Psi_+}.  }
 \label{fig:selbergmajorant}
\end{center}
\end{figure}

   Now take $\Phi=|\Psi_+|^2=\Psi_+^2$, which is \underline{positive} and has the same support as $\Psi_+$. 
   The Fourier transform $\^\Phi$ is the convolution
   \[
   \^\Phi(x) =\^\Psi_+ \conv \^\Psi_+ (x) 
   = \intinf  \^\Psi_+(y)\^\Psi_+(x-y)dy  .
   \]
   Since $\^\Psi_+ \geq \mathbf 1_{[-1,1]} \geq 0$, we have
$    \^\Phi(x) =\^\Psi_+ \conv \^\Psi_+ (x) 
     \geq  \mathbf 1_{[-1,1]}\conv \mathbf 1_{[-1,1]}(x).
$
But 
$  \mathbf 1_{[-1,1]}\conv \mathbf 1_{[-1,1]}(x) =  \max \left(2-|x|,0 \right)$
is a tent function, and in particular also majorizes 
the indicator function $ \mathbf 1_{[-1,1]}$ so that
  \[
  \^\Phi\geq  \max \left(2-|x|,0 \right) \geq \mathbf 1_{[-1,1]},
  \]
  as required.   
  \end{proof}

\end{document}